\documentclass[reqno, 11pt, a4paper, english]{smfart}
\usepackage{lmodern}
\usepackage{graphicx} 
\usepackage{mathtools} 
\usepackage{amsthm} 
\usepackage{amsfonts} 
\usepackage{smfthm}
\usepackage{amssymb}
\usepackage[T1]{fontenc}
\usepackage{mathrsfs} 
\usepackage{tikz-cd} 
\usepackage[sorting=nyt, maxnames=10]{biblatex} 
\usepackage{hyperref} 
\usepackage[a4paper, total={6in, 8in}]{geometry}
\usepackage{enumitem}
\usepackage[english,french]{babel}
\usepackage{csquotes}
\DeclareNameAlias{default}{family-given}

\DeclareMathOperator{\Q}{\mathbb{Q}}
\DeclareMathOperator{\R}{\mathbb{R}}
\DeclareMathOperator{\C}{\mathbb{C}}

\renewcommand{\P}{\mathbb{P}}

\renewcommand{\O}{\mathcal{O}}

\DeclareMathOperator{\Gr}{\operatorname{Gr}}
\DeclareMathOperator{\Jac}{\operatorname{Jac}}

\DeclarePairedDelimiter{\abs}{\lvert}{\rvert}

\DeclarePairedDelimiter{\norm}{\lVert}{\rVert}
\DeclarePairedDelimiter{\set}{\{}{\}}
\theoremstyle{plain}

\newtheorem{theorem}{Theorem}[section]
\newtheorem{lemma}[theorem]{Lemma}

\newtheorem{hypothesis}[theorem]{Hypothesis}
\newtheorem{proposition}[theorem]{Proposition}

\theoremstyle{definition}

\title{Counting rational points on non-uniruled surfaces}
\alttitle{Comptage de points rationnels de surfaces non uniréglées}
\date{\today}
\author{Lorenzo Andreaus}
	\address{Institut de Mathématiques de Jussieu - Paris Rive Gauche (IMJ-PRG)\\
		Campus des Grands Moulins\\
		Université Paris Cité - 8 Place Aurélie Nemours\\
		Paris, France}
	
	\email{lorenzo.andreaus@imj-prg.fr}	
\begin{document}

	\begin{abstract}
		Let $X\subseteq \P^n$ be a non-degenerate non-uniruled irreducible projective surface of degree $d$ defined over a number field $K$, and let $N_{X^{\prime}}(B)$ be the number of rational points of $X$ of height at most $B$ that do not lie on lines contained in $X$. Assuming a suitable hypothesis on the size of the rank of Abelian varieties, we show that $N_{X^{\prime}}(B)\ll_{K,n,d,\varepsilon}B^{\frac{n+1}{n}+\varepsilon}$ for any fixed $\varepsilon>0$. In particular, if $X\subseteq \P^3$ is a smooth projective surface of degree $d\ge 4$ defined over a number field $K$, then $N_{X^{\prime}}(B)\ll_{K,d,\varepsilon} B^{4/3+\varepsilon}$. This conditionally improves a bound of Salberger for $d=4$ and $d=5$. The proof, based on an argument of Heath-Brown, consists of cutting $X$ by projective hyperplanes and using a uniform version of Faltings's Theorem, due to Dimitrov, Gao, and Habegger, to bound the number of rational points on the hyperplane sections of $X$. 
	\end{abstract}

    \begin{altabstract}
    Soit $X\subseteq \P^n$ une surface projective irréductible non dégénérée et non uniréglée de degré $d$ définie sur un corps de nombres $K$, et soit $N_{X^{\prime}}(B)$ le nombre de points rationnels de $X$ de hauteur au plus $B$ qui ne sont pas situés sur des droites contenues dans $X$. Sous une hypothèse appropriée sur la taille du rang des variétés abéliennes, nous montrons que $N_{X^{\prime}}(B)\ll_{K,n,d,\varepsilon}B^{\frac{n+1}{n}+\varepsilon}$ pour tout $\varepsilon>0$ fixé. En particulier, si $X\subseteq \P^3$ est une surface projective lisse de degré $d\ge 4$ définie sur un corps de nombres $K$, alors $N_{X^{\prime}}(B)\ll_{K,d,\varepsilon} B^{4/3+\varepsilon}$. Ce résultat améliore conditionnellement une borne de Salberger pour $d=4$ et $d=5$. La démonstration, basée sur un argument de Heath-Brown, consiste à découper $X$ par des hyperplans projectifs et à utiliser une version uniforme du théorème de Faltings, due à Dimitrov, Gao et Habegger, pour majorer le nombre de points rationnels sur les sections hyperplanes de $X$.
    \end{altabstract}
\maketitle
\section{Introduction}
Let $K$ be a number field, and let $H_K$ be the height function in $\P^n$ given by:
\begin{eqnarray}\label{eq-height}
H_K([x_0\colon \cdots \colon x_n])=\prod_{v}\max\set{\norm{x_0}_v,\dotsc,\norm{x_n}_v},
\end{eqnarray}

where $v$ runs through the places of $K$ and $\norm{x}_v$ is the normalized absolute value associated with the place $v$ (this is the height function defined in \cite[Section B.2]{hindry}).

We consider a surface $X\subseteq \P^n$ of degree $d$, and we want to bound the number of $K$-rational points of bounded height on $X$. For this, define the counting function
\[N_X(B)\coloneqq \#\set{x\in X(K)\colon H_K(x)\le B}.
\]
If $X$ contains a line defined over $K$, then clearly $N_X(B)\gg B^2$. As we will prove in Lemma \ref{finite-number-2}, if $X$ is non-uniruled, then we have a finite number of such lines. If we remove these lines, we expect a smaller exponent. Let $X^{\prime}\subseteq X$ be the complement of the union of all lines on $X$. We define the counting function
\[
N_{X^{\prime}}(B)\coloneqq \#\set{x\in X^{\prime}(K)\colon H_K(x)\le B}.
\]

In our article, we will assume the following conjectural bound on the rank of Abelian varieties:

\begin{hypothesis}[Rank Hypothesis]
For any $K$-Abelian variety $A$ of fixed dimension $g$, let $N_A$ denote its conductor (that is, the ideal norm of the conductor ideal of $A$) and let $r_A$ denote its rank. Then we have $r_A=o(\log N_A)$ as $N_A\to \infty$.
\end{hypothesis}

Recall that a subvariety $X\subseteq \P^n$ is called \textbf{non-degenerate} if it is not contained in any projective hyperplane, and \textbf{uniruled} if there is a $K$-scheme $Y$ of dimension $\dim X-1$ and a dominant rational map $Y\times \P^1 \dashrightarrow X$.

Our main result is the following:

\begin{theorem}\label{main-result}
Assume that the Rank Hypothesis holds. Let $n\ge 3$, let $X\subseteq \P^n$ be a non-degenerate non-uniruled irreducible surface of degree $d$ defined over $K$, and let $\varepsilon>0$. Then we have $N_{X^{\prime}}(B)\ll_{K,n,d,\varepsilon}B^{\frac{n+1}{n}+\varepsilon}$.
\end{theorem}

Since a uniruled surface always has Kodaira dimension $-\infty$, the above theorem applies to surfaces of general type, properly elliptic surfaces, Abelian surfaces, hyperelliptic surfaces, Enriques surfaces and K3 surfaces. 

A classical argument using the Kodaira dimension and the Castelnuovo bound shows that, if $X$ as above is non-degenerate and non-uniruled, then we have $d\ge 2n-2$. For $n=3$, a smooth surface of degree $d$ is non-uniruled if and only if $d\ge 4$, so in particular we obtain:

\begin{theorem}\label{main-result-n=3}
Assume that the Rank Hypothesis holds. Let $X\subseteq \P^3$ be a smooth surface of degree $d\ge 4$ defined over $K$, and let $\varepsilon>0$. Then we have $N_{X^{\prime}}(B)\ll_{K,d,\varepsilon} B^{4/3+\varepsilon}$.
\end{theorem}

In the case $K=\Q$, we can compare our results with some unconditional bounds of Salberger (\cite{salberger}). If $X\subseteq\P^3$ is a smooth surface over $\Q$, Salberger shows in \cite[Theorem 0.5]{salberger} the bound

\begin{eqnarray}\label{eq-salberger-1}
N_{X^{\prime}}(B)\ll_{d} B^{\frac{3}{\sqrt{d}}}(\log B)^4+B.
\end{eqnarray}

In \cite[Theorem 0.6]{salberger}, he shows that if $X\subseteq \P^4$ is a smooth complete intersection of two hypersurfaces of degree $d_1$ and $d_2$ defined over $\Q$ and $d=d_1d_2$, then

\begin{eqnarray}\label{eq-salberger-2}
N_{X^{\prime}}(B)\ll_{d,\varepsilon}B^{\frac{3}{\sqrt{d}}+\varepsilon}+B,
\end{eqnarray}

and in general, if $X\subseteq \P^n$ is a geometrically integral projective surface of degree $d$ defined over $\Q$, he shows in \cite[Theorem 0.8]{salberger} the bound

\begin{eqnarray}\label{eq-salberger-3}
N_{X^{\prime}}(B)\ll_{n,d,\varepsilon} B^{\frac{3}{\sqrt{d}}+\varepsilon}+B^{\frac{3}{2\sqrt{d}}+2/3+\varepsilon}+B^{1+\varepsilon},
\end{eqnarray}

except for a particular class of rational surfaces of degree $d=4$.

An analysis of the exponents shows that, when $K=\Q$ and $n=3$, Theorem \ref{main-result-n=3} conditionally improves (\ref{eq-salberger-1}) for smooth surfaces of degree $d\in \set{4,5}$, and Theorem \ref{main-result} conditionally improves (\ref{eq-salberger-3}) for non-degenerate non-uniruled surfaces of degree $d\in\set{4,5}$.

When $K=\Q$ and $n=4$, Theorem \ref{main-result} conditionally improves (\ref{eq-salberger-3}) for non-degenerate non-uniruled surfaces of degree $d=6$. One can show that every smooth non-uniruled sextic surface in $\P^4$ is a $(2,3)$-complete intersection, so (\ref{eq-salberger-2}) covers all the smooth cases. Therefore, Theorem \ref{main-result} conditionally improves the known bound for singular non-uniruled sextic surfaces inside $\P^4$.

To prove Theorem \ref{main-result}, we will use a technique developed by Heath-Brown in \cite{heath-brown}. In this paper, Heath-Brown shows that, assuming the Rank Hypothesis for elliptic curves over $\Q$, we have, for any smooth cubic surface $X\subseteq \P^3$ and any fixed $\varepsilon>0$, a bound \[
N_{X^{\prime}}(B)\ll_{X,\varepsilon} B^{4/3+\varepsilon}.
\]
This technique was also applied by Glas and Hochfilzer to bound the number of rational points on del Pezzo surfaces of low degree over a global field, assuming the Rank Hypothesis for elliptic curves over number fields (\cite{glas-hochfilzer}).

We will cover the points of bounded height with projective hyperplanes of bounded height in the dual projective space $(\P^n)^{\vee}$. This can be done using a result of Schmidt (\cite{schmidt}). We then apply another result from the same paper to bound the number of such hyperplanes asymptotically, and proceed by estimating the number of points of bounded height in the intersection between $X$ and each of these hyperplanes. For irreducible components of these intersections of genus at least $2$, we use a uniform version of Faltings's Theorem, due to Dimitrov, Gao, and Habegger:

\begin{theorem}[Uniform Faltings's Theorem (\cite{mordell-lang})]\label{faltings}
For each integer $g\ge2$, there exists a constant $c(g)\ge 1$ depending on $[K\colon \Q]$ such that, for every smooth projective geometrically integral curve $C$ over $K$ of (geometric) genus $g$, we have
\[\#{C(K)}\le c(g)^{1+\rho(C)},
\]
where $\rho(C)$ denotes the rank of $\Jac(C)(K)$.
\end{theorem}

This is \cite[Theorem 1.1]{mordell-lang}. Compare it also with \cite[Theorem 4]{kuhne}, which eliminates the dependence on the degree $[K:\Q]$.

For irreducible components of genus $1$, we will use a bound on the number of rational points on smooth cubic curves due to Heath-Brown (\cite{heath-brown}), which is generalized to global fields by Glas and Hochfilzer (\cite{glas-hochfilzer}). Finally, we show that the curves of genus $0$ that appear are rare enough so they do not alter our estimation.

An advantage of our method is that it works over any number field, so we obtain a bound $O_{K,n,d,\varepsilon}(B^{\frac{n+1}{n}+\varepsilon})$ for every choice of $K$. So only the implicit constant depends on $K$, but the exponent is independent of the field. To do this, we use results of Paredes and Sasyk (\cite{paredes}) and Glas and Hochfilzer (\cite{glas-hochfilzer}). However, notice that we use the height $H_K$ of the field $K$, while in \cite{paredes} they use the absolute height $H=H_K^{1/[K:\Q]}$. 

We hope to generalize this method in a future paper to obtain results about higher-dimensional varieties.

\section{Notations and Remarks}

Let $K$ be a number field. For a place $v$ of $K$ lying over the rational place $p$ (where $p$ is a prime or $\infty$), we define $\abs{\cdot}_v\colon K\to \R_{+}$ as the absolute value on $K$ whose restriction to $\Q$ is $\abs{\cdot}_p$. We define $\norm{\cdot}_v\coloneqq \abs{\cdot}_v^{[K_v\colon \Q_p]}$, where $K_v$ and $\Q_p$ denote the completions of $K$ and $\Q$ with respect to $v$ and $p$. Using the product formula, one can see that $H_K$, defined as in (\ref{eq-height}), is well-defined in the projective space. We can also define the absolute height $H(x)\coloneqq H_K(x)^{1/[K\colon \Q]}$, and its logarithmic absolute height $h(x)\coloneqq \log H(x)$. The absolute heights $H(x)$ and $h(x)$ are independent of the number field containing the coordinates of $x$.

Given a nonzero vector $w\in K^m$, its height $H_K(w)$ is defined as the height of its projection $[w]\in \P^{m-1}(K)$.

Some articles use a slightly different definition of $H_K$, taking the Euclidean norm at the infinite places instead of the maximum norm. However, since both norms differ by a multiplicative constant, this distinction will not matter in what follows, so we shall ignore it.

Given a polynomial $P(x_0,\dotsc,x_n)\in K[x_0,\dotsc, x_n]$, we define its height $H_K(P)$ as the height over $K$ of the projective point whose entries are the coefficients of $P$. We can also consider its logarithmic absolute height $h(P)$ similarly as above.

We will also need the notion of the (absolute) height of an integral subvariety $X\subseteq \P^n$. We follow \cite{remond}. Let $F$ be the Chow form of $X$ and $S_{n+1}\subseteq \C^{n+1}$ be the unit sphere. Let $L$ be a number field containing all the coefficients of $F$. For any Archimedean place $v$ of $L$, define $M_v(F)$ by
\[
\log M_v(F)=(\dim X+1)(\deg X)\sum_{j=1}^n\frac{1}{2j}+\int_{S_{n+1}^{\dim X+1}}\log\abs{F_v}\sigma,
\]
where $\sigma$ is the invariant measure of total mass $1$ over $S_{n+1}^{\dim X+1}$. For $v$ finite, we define $M_v(F)\coloneqq \max_a\abs{a}_v$, where $a$ ranges over the coefficients of $F$. Then the height of $X$ is, by definition:
\[
h(X)\coloneqq \sum_v \frac{[L_v\colon \Q_p]}{[L\colon \Q]}\log M_v(F).
\]

Curves, surfaces and varieties in general are always considered to be reduced.

By genus, we always mean the geometric genus, unless explicitly stated.

We will use the Vinogradov notations $\ll, \gg$ and $\gg\ll$, and the big O and little-o notations. The implied constant depends on the subscript variables. 

\section{The Rank Hypothesis}

In this section, we will discuss in more detail the Rank Hypothesis, which states that $r_A=o(\log N_A)$ for any $K$-Abelian variety $A$ of rank $r_A$ and conductor $N_A$. For $K$ and $g$ fixed, this means that for all $\varepsilon>0$, there exists a constant $N_0=N_0(K,g,\varepsilon)>0$ such that for any Abelian variety $A$ of dimension $g$ defined over $K$, $r_A\le \varepsilon\log N_A$ if $N_A\ge N_0$. Since the number of Abelian varieties of dimension $g$ over $K$ of bounded conductor is finite by the Shafarevich Conjecture proved by Faltings (\cite{faltings1983endlichkeitssatze}), we can write this in the following way: there exists a constant $C_0=C_0(K,g,\varepsilon)$ such that 
\begin{eqnarray}\label{eq-rank-hypothesis}
1+r_A\le \varepsilon \log N_A+C_0,
\end{eqnarray} for any Abelian variety $A$ of dimension $g$ over $K$ (the extra $1$ on the left-hand side is just a convenience to ease the application of the Uniform Faltings's Theorem).

The following analytical argument suggests the plausibility of the Rank Hypothesis:

Assuming modularity for the $g$-dimensional Abelian variety $A$ over the number field $K$, this hypothesis would hold for $A$ by a combination of the Generalized Riemann Hypothesis and the Birch and Swinnerton-Dyer conjecture. Indeed, let $L(A,s)$ be the $L$-function associated with the Abelian variety $A$. Let $\delta_{K/\Q}$ be the discriminant of $K$. Then, following \cite{serre1969facteurs}, if we define
\[
\Lambda(A,s)\coloneqq (\abs{\delta_{K/\Q}}^{2g}N_A)^{s/2}((2\pi)^{-s}\cdot\Gamma(s))^{g[K\colon \Q]}\cdot L(A,s),
\]
by the modularity of $A$ the function $\Lambda(A,s)$ extends to an entire function and satisfies the functional equation
\[
\Lambda(A,2-s)=\pm\Lambda(A,s).
\]
Assuming the Generalized Riemann Hypothesis and proceeding by classical analytic arguments, this functional equation would give us
\[\operatorname{ord}_{s=1}(L(A,s))=O\left(\frac{\log N_A}{\log\log N_A}\right)=o(\log N_A).
\]
Assuming the Birch and Swinnerton-Dyer conjecture, we obtain 
\[
\operatorname{ord}_{s=1}(L(A,s))=r_A,
\]
which gives the Rank Hypothesis. See \cite[Proposition II.1]{mestre} for a detailed proof of these implications in the case of the $L$-function of a modular form.

We now cite some known results that go in the direction of the Rank Hypothesis. By \cite[Theorem 1]{ooe}, for an Abelian variety $A/K$ we have, unconditionally,
\[r_A=O_{K,\dim A}(\log N_A+1).\]
In the case of elliptic curves, more evidence is known. Notice that all elliptic curves over the rational numbers are modular by \cite{breuil}. If we assume that an elliptic curve $E/\Q$ has at least one rational point of order $2$, then it is known that (see \cite[Section 6]{heath-brown})
\[
r_E=O\left(\frac{\log N_E}{\log\log N_E}\right)=o(\log N_E).
\]

For function fields of characteristic $p\ge 5$, the Rank Hypothesis for elliptic curves was proved by Brumer (\cite{brumer}).

\section{Counting rational points on curves}

Since our strategy is to cut our surface by projective hyperplanes and count the number of rational points on each of the intersection curves, we start with some results about rational points on curves. To deal with curves of genus $0$ and non-geometrically integral curves, we use the following result:

\begin{theorem}\label{thm-curves-1}
Let $C\subseteq \P^n$ be an integral curve of degree $e$. Then, for any $\varepsilon>0$, $N_C(B)\ll_{K,n,e,\varepsilon}B^{2/e+\varepsilon}$. If $C$ is not geometrically integral, then $N_C(B)\ll_{K,n,e} 1$.
\end{theorem}
\begin{proof}
The first part of the theorem is \cite[Theorem 1.8]{paredes}. The arguments in \cite[Theorem 1.17]{salberger} to deal with the non-geometrically integral case work for number fields.
\end{proof}

If $g(C)\ge 1$ and we assume the Rank Hypothesis, we can obtain a much better bound. We will use \cite[Proposition 1.4]{glas-hochfilzer} to deal with the genus $1$ case:

\begin{proposition}\label{genus1}
Assume that the Rank Hypothesis holds for elliptic curves, and let $E\subseteq \P^m$ be a smooth curve of genus $1$ and degree $e$. Then, for any $\varepsilon>0$, we have $N_E(B)\ll_{K,m,e,\varepsilon} B^{\varepsilon}$.
\end{proposition}

We also need the following dichotomy, which is a corollary of \cite[Lemme 3.1]{remond2} applied to curves:

\begin{lemma}\label{dichotomy-curves}
Let $C\subseteq \P^n$ be a geometrically integral curve of degree $e$, and let $B>1$. Then, either $N_C(B)\ll_{n,e} 1$ or $h(C)\ll_{n,e}\log B+1$. 
\end{lemma}

\begin{theorem}\label{thm-curves-2}
Assume that the Rank Hypothesis holds. Let $C\subseteq \P^n$ be an irreducible curve of degree $e$ and genus $g\ge 1$. Then $N_C(B)\ll_{K,n,e,\varepsilon} B^{\varepsilon}$.
\end{theorem}
\begin{proof}
By the second part of Theorem \ref{thm-curves-1}, we can assume that $C$ is geometrically integral. By the arguments in \cite[Section 3]{algebraic-varieties}, we can find a linear projection $C\to \P^2$ that sends $C$ birationally to a geometrically integral plane curve $C^{\prime}\subseteq \P^2$ of degree $e$ such that $N_C(B)\le eN_{C^{\prime}}(c_0B)$, where $c_0\ll_{K,n,e}1$. Thus, we can assume that $C$ is a plane curve. By Lemma \ref{dichotomy-curves}, we can also assume that $h(C)\ll_e\log B+1$. Let $\tilde{C}$ be the normalization of $C$, and $p_a(C)$ denote the arithmetic genus of $C$. By \cite[Example V.3.9.2 and Exercise I.7.2(b)]{hartshorne},
\begin{eqnarray}\label{eq-genus}
g(\tilde{C})=p_a(C)-\sum_{P}\frac{m_P(m_P-1)}{2}=\frac{(e-1)(e-2)}{2}-\sum_P\frac{m_P(m_P-1)}{2},
\end{eqnarray}

where the sum is taken over all singular points of $C$, including infinitely near singular points, and $m_P$ denotes the multiplicity of $C$ at $P$. In particular, as $g\ge 0$, this shows that the number of singularities of $C$ is at most $\frac{(e-1)(e-2)}{2}$, which is $O_e(1)$. Also notice that $g\le \frac{(e-1)(e-2)}{2}$, so we can replace a dependence on $g$ by a dependence on $e$. This inequality also implies that $e\ge 3$, since $g\ge 1$.

Applying \cite[Théorème 1.5]{remond}, we can find an embedding $\tilde{C}\xhookrightarrow{}\P^3$ such that $\deg \tilde{C}\ll_e 1$ and $h(\tilde{C})\ll_{K,e} h(C)+1$. The proof of \cite[Théorème 1.5]{remond} also shows that we can compare the heights of points in $C\xhookrightarrow{}\P^2$ and in $\tilde{C}\xhookrightarrow{}\P^3$ using the normalization map $\nu\colon \tilde{C}\to C$. Indeed, the construction is explicit. If $C$ is smooth, then we use the embedding $\tilde{C}=C\xhookrightarrow{}\P^2\xhookrightarrow{}\P^3$, which preserves the heights of points. So we can suppose that $C$ is singular. Since a singular plane cubic has genus $0$ and $g\ge 1$, we have $e\ge 4$.

Denote by $x_0$, $x_1$ and $x_2$ the coordinates in $\P^2$, $t\coloneqq x_0/x_2$ and $u\coloneqq x_1/x_2$. First, if needed, we make a linear change of variables with coefficients of height controlled by $e$, so that $C$ has no singularities on the divisor $D\coloneqq V(x_2)$ and that $[0\colon 1\colon 0]\not\in C$. Let $F(t,u)\in K[t,u]$ be the polynomial defining $C$. Let $\Delta(t)\coloneqq \operatorname{disc}_u F(t,u)$ be the discriminant of $F$ with respect to $u$, where $F(t,u)$ is viewed as a monic polynomial of degree $e$ in $u$ with coefficients in $K[t]$ (we can normalize $F$ so that this polynomial is monic, since $[0\colon 1\colon 0]\not \in C$). We can show that $\deg \Delta\le e(e-1)$ and $h(\Delta)\ll_e h(C)+1$.

Then we can select three polynomials $P_1(t,u), P_2(t,u)$ and $P_3(t,u)$ satisfying
\[
\deg P_i\le \deg\Delta+e-2\le e^2-2,\text{ and }h(P_i)\ll_{K,e} h(C)+1,
\]
such that $1,\Delta^{-1}P_1, \Delta^{-1}P_2$ and $\Delta^{-1}P_3$ are sections of the very ample divisor $(e-2)\nu^*D$ defining an embedding $\tilde{C}\xhookrightarrow{}\P^3$. Let $\mu\coloneqq\max\set{\deg \Delta, \deg P_1,\deg P_2,\deg P_3}$. This embedding is given by degree-$\mu$ polynomials:
\[
[x_0\colon x_1\colon x_2]\mapsto [x_2^{\mu}\Delta(x_0/x_2)\colon x_2^{\mu}P_1(x_0/x_2,x_1/x_2)\colon x_2^{\mu}P_2(x_0/x_2,x_1/x_2)\colon x_2^{\mu}P_3(x_0/x_2,x_1/x_2)].
\]
This shows that, for every $x\in \tilde{C}(K)$ lying outside a set of $O_e(1)$ points,
\begin{eqnarray*}
h_{\tilde{C}}(x)&\le& \mu h_C(\nu(x))+O_{K,e}(h(C)+1)\\
&\le&(e^2-2) h_C(\nu(x))+O_{K,e}(h(C)+1),
\end{eqnarray*}
where $h_{\tilde{C}}$ and $h_C$ denote the logarithmic absolute heights in $\tilde{C}\xhookrightarrow{}\P^3$ and $C\xhookrightarrow{}\P^2$. Since we assumed $h(C)\ll_e\log B+1$, we can find a positive constant $a(K,e)$ such that
\[
H_K(x)\le H_K(\nu(x))^{e^2-2}B^{a(K,e)},
\]
therefore $N_C(B)\le N_{\tilde{C}}(B^{e^2-2+a(K,e)})+O_e(1)$.

So it suffices to prove that $N_{\tilde{C}}(B)\ll_{K,n,e,\varepsilon}B^{\varepsilon}$. From now on, we assume that $C\subseteq \P^3$ is a smooth curve. The case $g=1$ then follows immediately from the previous proposition for $m=3$.

If $g\ge 2$, we can apply the Uniform Faltings's Theorem (Theorem \ref{faltings}) to find
\[\# C(K)\le c(g)^{1+\rho(C)}.
\]

Let $J$ be the Jacobian of $C$. By the Rank Hypothesis (\ref{eq-rank-hypothesis}), $1+\rho(C)\le \varepsilon \log N_J+C_0(K,g,\varepsilon)$. By \cite[Théorème 1.4]{remond}, $\log N_J\ll_{K,e,g} h(C)+1$, and since we assumed that $h(C)\ll_{e}\log B+1$, we can find constants $c_1(K,e,g),c_2(K,e,g,\varepsilon)>0$ such that 
\[1+\rho(C)\le c_1(K,e,g)\varepsilon\log B+c_2(K,e,g,\varepsilon).
\]
Thus, by the Uniform Faltings's Theorem:
\begin{eqnarray*}
\#C(K)\le c(g)^{1+\rho(C)}\le c(g)^{c_2(K,e,g,\varepsilon)}\cdot c(g)^{c_1(K,e,g)\varepsilon\log B}\ll_{K,e,g,\varepsilon} B^{\varepsilon c^{\prime}(K,e,g)},
\end{eqnarray*}
where $c^{\prime}(K,e,g)\coloneqq c_1(K,e,g)\log c(g)$. Since $\varepsilon>0$ is arbitrary, we obtain
\[\# C(K)\ll_{K,e,g,\varepsilon} B^{\varepsilon}.
\]
We conclude the proof by replacing the dependence on $g$ by a dependence on $e$.
\end{proof}

\section{Counting rational points on surfaces}

Let $X\subseteq \P^n$ be a non-degenerate non-uniruled irreducible surface of degree $d$ defined over $K$. If $X$ is not geometrically integral, then we have a decomposition $X_{\overline{\Q}}\coloneqq X\times_K \overline{\Q}=X_1\cup\cdots\cup X_r$ into $r\ge 2$ irreducible components of degree $d/r$, which are Galois conjugate. Since a $K$-rational point is invariant under the action of $\operatorname{Gal}(\overline{\Q}/K)$, it gives rise to a $\overline{\Q}$-point which lies on $\overline{Y}\coloneqq X_1\cap\cdots\cap X_r$. Notice that $\overline{Y}$ is fixed under $\operatorname{Gal}(\overline{\Q}/K)$, hence there exists a closed subscheme $Y\subseteq X$ defined over $K$ such that $\overline{Y}=Y\times_K\overline{\Q}$ and such that $X(K)=Y(K)$. Since $\overline{Y}\subseteq X_1\cap X_2$, we see that $\dim Y\le 1$ and that the sum of the degrees of the irreducible components of $Y$ is at most $d^2/r^2\le d^2$ by Bézout's Theorem. Applying Theorem \ref{thm-curves-1} to each integral component of $Y$ of degree $e\ge 2$ (since we are not counting the lines), we conclude that $N_{X^{\prime}}(B)\ll_{K,n,d,\varepsilon}B^{1+\varepsilon}$, proving Theorem \ref{main-result} for non-geometrically integral surfaces. Notice that in this case the exponent is even smaller and that we do not need to assume the Rank Hypothesis. So, from now on we may assume that $X$ is geometrically integral.

Let $V$ be an $m$-dimensional $K$-vector space, and let $H_K$ be the height on $\P(V)\cong \P^{m-1}$. Let $1\le k\le m$. We start by defining a height in the Grassmannian
\[
\Gr_K(k,m)\coloneqq\set{W\subseteq V\colon W\text{ is a $K$-linear subspace of dimension }k}.
\]
In our argument, we will only use $\Gr_K(m-1,m)$, which is the dual projective space $(\P^{m-1})^{\vee}$. However, the more general concept of a Grassmannian will be useful if one wants to generalize this argument to varieties of higher dimension. 

Let $W\subseteq V$ be a $k$-dimensional subspace. Let $w_1,\dotsc, w_k$ be any basis of $W$. Consider the exterior product $w_1\wedge \cdots \wedge w_k\in \bigwedge^k V\cong K^{\binom{m}{k}}$. We define $H_K(W)\coloneqq H_K(w_1\wedge\cdots\wedge w_k)$, where we are taking the height on $\P(\bigwedge^k V)\cong \P^{\binom{m}{k}-1}$. It is possible to show that this height does not depend on the choice of basis, so $H_K(W)$ is well-defined. For more on the heights of subspaces, see \cite{schmidt}. 

Returning to our original problem, let $x\in X(K)$ with $H_K(x)\le B$. We want to find a projective hyperplane $\Pi\cong \P^{n-1}$ of small height containing $x$. This is equivalent to finding an $n$-dimensional subspace of an $(n+1)$-dimensional space that contains a given vector $x$. By \cite[Theorem 2]{schmidt}, there exists such a hyperplane $\Pi$ with $H_K(\Pi)\ll_{K,n} B^{1/n}$. This shows that, for a given $B>0$, the set 
\[
\set{x\in X(K)\colon H_K(x)\le B}
\]
can be covered by hyperplanes of height $O_{K,n}(B^{1/n})$. Using now \cite[Theorem 3]{schmidt} (or the more precise version \cite[Theorem 1]{thunder}), we see that the number of such hyperplanes is
\[O_{K,n} \big(\big(B^{1/n}\big)^{n+1}\big)=O_{K,n}(B^{\frac{n+1}{n}}).
\]

We proceed by counting the number of rational points in the intersection of $X$ with each such hyperplane. Let $\Pi$ be one of the hyperplanes of bounded height as above. Since $X$ is non-degenerate, $X\cap \Pi$ is a degree-$d$ curve. Let $C$ be an irreducible degree-$e$ component of this curve, with $1\le e\le d$. If $g(C)\ge 1$ or $C$ is not geometrically integral, by Theorem \ref{thm-curves-2} and Theorem \ref{thm-curves-1} we have $N_C(B)\ll_{K,n,d,\varepsilon}B^{\varepsilon}$, so the total contribution coming from these curves is $O_{K,n,d,\varepsilon}(B^{\frac{n+1}{n}+\varepsilon})$.

All that is left is to bound the number of rational points on geometrically integral curves of genus $0$ that lie inside $X$. Let $C\subseteq X$ be a geometrically integral curve of genus $0$ and degree $e$, with $1\le e\le d$. Then $C$ is a geometrically rational curve. If $e=1$, then $C$ is a line, and by definition it does not contribute to the calculation of $N_{X^{\prime}}(B)$. Suppose $e\ge 2$. By Theorem \ref{thm-curves-1},
\[
N_C(B)\ll_{K,n,d,\varepsilon} B^{2/e+\varepsilon}\le B^{1+\varepsilon}.
\]
Since the contribution of a single geometrically rational curve of degree $e\ge 2$ is less than $O_{K,n,d,\varepsilon}(B^{\frac{n+1}{n}+\varepsilon})$, and since $2\le e\le d$, we conclude the proof of Theorem \ref{main-result} if we show that, for any $e\ge 1$, a non-uniruled irreducible surface of degree $d$ contains at most $O_{n,d,e}(1)$ geometrically rational curves of degree $e$. In fact, we can show the concrete bound $d^{(n+1)(e+1)}$. We start with the following lemma, where we will use that a finite union of proper subspaces of a vector space does not cover the whole space, since we are working over an infinite field.

\begin{lemma}\label{N-forms}
Let $W$ be a space of degree-$d$ forms on a projective space $\P^N$, and let $Y=V(W)$ be its vanishing set. Then we can find $M+1\le N+1$ forms $H_0,\dotsc, H_M\in W$, such that $Y=V(H_0,\dotsc, H_M)$.
\end{lemma}
\begin{proof}
We may assume $W\ne 0$. Choose a non-zero element in $W$ and call it $H_0$. Define $Y_1\coloneqq V(H_0)$. We have $Y\subseteq Y_1$ and $\dim Y_1=N-1$. If $Y=Y_1$, we are done. Otherwise, we can find some $H_1\in W$ that does not vanish identically on any of the finitely many irreducible components of $Y_1$ not contained in $Y$. Let $Y_2\coloneqq V(H_0,H_1)$. Then $Y\subseteq Y_2$, and every component of $Y_2$ not contained in $Y$ has dimension at most $N-2$. Continuing in this way, we find $H_0,\dotsc, H_M$, with $M\le N$, such that $Y=Y_{M+1}\coloneqq V(H_0,\dotsc, H_M)$, as we wanted.
\end{proof}

Finally, we have:

\begin{lemma}\label{finite-number-2}
Let $X\subseteq \P^n$ be a non-uniruled geometrically integral surface of degree $d$ defined over a number field $K$. Fix a positive integer $e$. Then $X_{\overline{\Q}}$ contains at most $d^{(n+1)(e+1)}$ irreducible rational curves of degree $e$.
\end{lemma}
\begin{proof}
This proof is entirely geometrical, so we can work over $\overline{\Q}$ and write $X$ for $X_{\overline{\Q}}$. We first parametrize rational curves of degree $e$ as follows: if $C\subseteq \P^n$ is an irreducible rational curve of degree $e$ with normalization $\nu\colon\P^1\rightarrow C$, we have
$\nu^*(\O_{\P^n}(1))=\O(e)$. Hence, there exist degree-$e$ forms in two variables $f_0,\dotsc, f_n$ such that $\nu$ is given by 
\[
\nu([x_0\colon x_1])=[f_0(x_0, x_1)\colon\cdots\colon f_n(x_0,x_1)].
\]
Let $N\coloneqq (n+1)(e+1)-1$, and consider the projective space $\P^N$ of $(n+1)$-tuples $f=[f_0\colon\cdots\colon f_n]$, where each $f_j$ is a degree-$e$ form in two variables. Let $U\subseteq \P^N$ be the open subset of tuples without a common zero in $\P^1$, so that an element $f\in U$ corresponds to a morphism $f\colon \P^1\to \P^n$. Let
\[
U_X\coloneqq \set{f\in U\colon f(\P^1)\subseteq X}
\]
be the subset of $U$ consisting of morphisms $f\colon \P^1\to X$.

Let $F_1,\dotsc, F_s$ be a basis of the degree-$d$ forms vanishing on $X$. For each $i$ and each $f\in \P^N$, $F_i\circ f$ is a degree-$de$ form in two variables whose $de+1$ coefficients are degree-$d$ forms $G_{i,0}(f),\dotsc, G_{i,de}(f)$ in the coordinates of $\P^N$. Define $V_X\subseteq \P^N$ as the vanishing set
\[
V_X\coloneqq V(\set{G_{i,j}\colon 1\le i\le s,~0\le j\le de}).
\]
It is well known that $X$ is the vanishing set of the degree-$d$ forms vanishing on $X$, so $X=V(F_1,\dotsc, F_s)$. From this, it is easy to see that $U_X=U\cap V_X$. Since $V_X\subseteq \P^N$ is the vanishing set of a space of degree-$d$ forms, by Lemma \ref{N-forms} there are $M+1\le N+1$ degree-$d$ forms $H_0,\dotsc, H_M$ such that $V_X=V(H_0,\dotsc, H_M)=\bigcap_{i=0}^M V(H_i)$. Hence, by Bézout's Theorem (\cite[Example 8.4.6]{fulton}), $V_X$ has at most $d^{M+1}\le d^{N+1}=d^{(n+1)(e+1)}$ irreducible components. Since $U_X\subseteq V_X$ is an open subset, the same holds for $U_X$. 

Denote by $\operatorname{Irr}(U_X)$ the set of irreducible components of $U_X$, by $\operatorname{Rat}_X$ the set of irreducible rational curves contained in $X$ and by $\operatorname{Rat}_X^e\subseteq \operatorname{Rat}_X$ the subset of degree-$e$ curves. Let $Z\in \operatorname{Irr}(U_X)$. We claim that all $f\in Z$ have the same image $f(\P^1)\subseteq X$, which is in $\operatorname{Rat}_X$. First, since $U_X$ consists of non-constant maps (since they have degree $e>0$), each $f(\P^1)\subseteq X$ is an irreducible rational curve. Suppose by contradiction that there are $f_1,f_2\in Z$ such that $f_1(\P^1)\ne f_2(\P^1)$. By \cite[Lemma of Section 6]{mumford}, we can find an irreducible curve $Z^{\prime}\subseteq Z$ passing through $f_1$ and $f_2$. Consider the evaluation morphism $\operatorname{ev}\colon U\times \P^1\to \P^n$. Then $\operatorname{ev}(Z^{\prime}\times \P^1)$ is irreducible and contains two distinct irreducible curves, so it has dimension at least $2$. Thus, it is dense in $X$. This shows that $X$ is uniruled, contradicting our hypothesis. So we can define a function $\operatorname{Rat}\colon \operatorname{Irr}(U_X)\to \operatorname{Rat}_X$ by $\operatorname{Rat}(Z)=f(\P^1)$ for any $f\in Z$.

Now, let $C\in \operatorname{Rat}_X^e$. If $Z\in \operatorname{Irr}(U_X)$ is an irreducible component that contains the normalization of $C$, then $\operatorname{Rat}(Z)=C$. This shows the inclusion $\operatorname{Rat}_X^e\subseteq \operatorname{Rat}(\operatorname{Irr}(U_X))$. We conclude that $\#\operatorname{Rat}_X^e\le \#\operatorname{Irr}(U_X)\le d^{(n+1)(e+1)}$.
\end{proof}

This concludes the proof of Theorem \ref{main-result}.

\printbibliography
\end{document}